\documentclass[11pt,leqno]{article}
\usepackage{amssymb,amsfonts,amsmath,amsthm,amscd,mathrsfs}
\usepackage{psfrag}
\def\IP{{\mathbb P}}
\def\IR{{\mathbb R}}
\def\IN{{\mathbb N}}

\def\IQ{{\mathbb Q}}
\def\IZ{{\mathbb Z}}
\def\IG{{\mathbb G}}

\def\fr{\mbox{\footnotesize $\dis\frac{1}{2}$}}
\def\n{\noindent}
\def\dis{\displaystyle}
\def\r{\rightarrow}
\def\ov{\overline}
\def\wt{\widetilde}
\def\o{\omega}

\def\cD{{\cal D}}
\def\cH{{\cal H}}
\def\cN{{\cal N}}
\def\cM{{\cal M}}
\def\cZ{{\cal Z}}

\def\H{{\rm Homeo}^+}
\def\Hom{{\rm Hom}}

\newcommand{\textus}[1]{\textbf{#1}}

\dimendef\dimen=0

\newtheorem{theorem}{Theorem}[section]
\newtheorem{lemma}[theorem]{Lemma}
\newtheorem{corollary}[theorem]{Corollary}
\newtheorem{proposition}[theorem]{Proposition}
\newtheorem{definition}[theorem]{Definition}
\newtheorem{remark}[theorem]{Remark}
\newtheorem{example}[theorem]{Example}

\begin{document}

\noindent
\begin{center}
{\LARGE  \bf An extension criterion for \\[0.5ex] lattice actions on the circle}
\end{center}

\begin{center}
Marc Burger \\
Forschungsinstitut f\"ur Mathematik\\
ETH-Z\"urich 
\end{center}

\hfill {\it To Bob Zimmer}

\bigskip
\section{Introduction}

Let $\Gamma < G $ be a lattice in a locally compact second countable group $G$. The aim of this paper is to establish a necessary and sufficient condition for a $\Gamma$-action by homeomorphisms of the circle to extend continuously to $G$. This condition will be expressed in terms of the real bounded Euler class of this action. Combined with classical vanishing theorems in bounded cohomology, one recovers rigidity results of Ghys, Witte-Zimmer, Navas and Bader-Furman-Shaker in a unified manner. For a survey of various approaches to the problem of classifying lattice actions on the circle we refer to the article of Witte Morris \cite{W08} in this volume.

\medskip
Let $\H(S^1)$ be the group of orientation preserving homeomorphisms of the circle and $e \in H^2 (\H(S^1), \IZ)$ the Euler class; recall that $e$ corresponds to the central extension defined by the universal covering of $\H(S^1)$. The Euler class admits a representing cocycle which is bounded and this defines a bounded class $e^b \in H^2_b(\H(S^1), \IZ)$ called the bounded Euler class. The relevance of bounded cohomology to the study of group actions on the circle comes from a result of Ghys \cite{Gh87}, namely that the bounded Euler class $\rho^*(e^b) \in H^2_b(\Gamma, \IZ)$ of an action $\rho: \Gamma \rightarrow \H(S^1)$ determines $\rho$ up to quasi-conjugation; a quasi-conjugation is a self map of the circle which is weakly cyclic order preserving and in particular not necessarily continuous, see Section 3 for details. If $e^b_{\IR}$ denotes then the bounded class obtained by considering the bounded cocycle defining $e^b$ as real valued, we call the invariant $\rho^*(e^b_{\IR}) \in H^2_b(\Gamma, \IR)$ the real bounded Euler class of $\rho$. From this point of view we have the following dichotomy (see Proposition \ref{prop3.2}):

\medskip\n
(E) $\rho^*(e^b_\IR) = 0$:  in this case, $\rho$ is quasi-conjugated to an action of $\Gamma$ by rotations; as far as the extension problem is concerned, it reduces to the properties of the restriction map
\begin{equation*}
\Hom_c(G, \IR / \IZ) \rightarrow \Hom(\Gamma, \IR / \IZ) \,.
\end{equation*}

\n
(NE) $\rho^*(e^b_\IR) \not= 0$: in this case, $\rho$ is quasi-conjugated to a minimal unbounded action, that is, every orbit is dense and the group of homeomorphisms $\rho(\Gamma)$ is not equicontinuous.

\medskip
In the first case (E) we call $\rho$ elementary and in the second (NE) non-elementary; non-elementary actions are our main object of study in this paper.

\medskip
Concerning the extension problem, an issue which has to be taken care of is the existence of a non-trivial centralizer of the action under consideration. This is illustrated by the following

\begin{example}\label{exam1.1} \rm
Let $\Gamma < {\rm PSL}(2, \IR)$ be a lattice which is non-uniform and torsion free. Since $\Gamma$ is a free group we can lift the identity to a homomorphism $\rho_k$: $\Gamma \rightarrow {\rm PSL}(2, \IR)_k \subset \H(S^1)$ into the $k$-fold cyclic covering of PSL$(2, \IR)$, and this for every $k \ge 1$. In this way we get an action which is minimal, unbounded, but for $k \ge 2$ does not extend continuously to PSL$(2, \IR)$. If $\Gamma$ is torsion free co-compact this construction applies provided $k$ divides the Euler characteristic of $\Gamma$ which is always the case for $k = 2$.
\end{example}

Thus given a minimal unbounded action one is lead to consider its topological $S^1$-factors; those are easily classified and in particular there is, up to conjugation, a unique factor
\begin{equation*}
\rho_{sp}: \Gamma \rightarrow \H(S^1)
\end{equation*}

\n
which is strongly proximal (Proposition \ref{prop3.7}). 

\medskip
This relies on arguments of Ghys which establish that the centralizer of $\rho(\Gamma)$ is a finite cyclic group; the strongly proximal quotient is then obtained by passing to the quotient by this cyclic group.

\medskip
Our main result is:
\begin{theorem}\label{theo1.2}
Let $\Gamma < G$ be a lattice in a locally compact, second countable group $G$ and
\begin{equation*}
\rho: \Gamma \rightarrow \H(S^1)
\end{equation*}
be a minimal unbounded action. Then the following are equivalent:

\begin{itemize}
\item[{\rm 1)}] The bounded real Euler class $\rho^*(e^b_{\IR})$ of $\rho$ is in the image of the restriction map
\begin{equation*}
H^2_{bc} (G, \IR) \rightarrow H^2_b (\Gamma, \IR)\,.
\end{equation*}

\item[{\rm 2)}] The strongly proximal factor $\rho_{sp}$ of $\rho$ extends continuously to $G$.
\end{itemize}
\end{theorem}

The main ingredient in the proof of Theorem \ref{theo1.2} is a result which, for any countable group $\Gamma$, characterizes the bounded classes in $H^2_b(\Gamma, \IR)$ obtained from minimal strongly proximal actions in terms of certain cocycles defined on appropriate Poisson boundary of $\Gamma$; see Theorem \ref{theo4.5} in Section 4, where the result is proven in the more general context of locally compact, second countable groups. This leads to results of independent interest concerning the extent to which an action is determined by its real bounded Euler class and to the question of the possible values of its norm. These results are summarized in the following theorem and corollaries.

\begin{theorem}\label{theo1.3}
Let $\Gamma$ be a countable group.

\begin{itemize}
\item[{\rm 1)}] For a homomorphism $\rho: \Gamma \rightarrow \H(S^1)$ we have
\begin{equation*}
\| \rho^* (e^b_\IR)\| \le \fr
\end{equation*}
with equality if and only if $\rho$ is quasi-conjugated to a minimal strongly proximal action.

\item[{\rm 2)}] If two minimal strongly proximal actions $\rho_1,\rho_2$ are not conjugated, then
\begin{equation*}
\| \rho_1^*(e^b_\IR) - \rho_2^* (e^b_\IR) \| = 1\,.
\end{equation*}
\end{itemize}
\end{theorem}

\begin{remark}\label{rem1.4} \rm 

\medskip\n
(1) The first assertion in Theorem \ref{theo1.3} echoes results obtained in \cite{B-I-W07b} concerning tight homomorphisms with values in a Lie group of hermitian type. 

\medskip\n
(2) Let ESP$(\Gamma) \subset H^2_b(\Gamma,\IR)$ denote the subset consisting of the real bounded Euler classes of minimal strongly proximal $\Gamma$-actions and $\IZ[{\rm ESP}(\Gamma)]$ its $\IZ$-span. We will show (see Section 5) that the norm takes half integral values on $\IZ[{\rm EPS}(\Gamma)]$. In this context the following question arises, namely if $\IZ[{\rm SP}(\Gamma)]$ denotes the free Abelian group on the set of conjugacy classes of minimal strongly proximal actions, can one determine the kernel of the homomorphism
\begin{align*}
\IZ[SP(\Gamma)] & \rightarrow \IZ [{\rm ESP}(\Gamma)]
\\[1ex]
\Sigma n_i [\rho_i] & \rightarrow \Sigma n_i \,\rho_i(e^b_\IR)
\end{align*}

\n
and what is its significance for the dynamics of $\Gamma$-actions on $S^1$?
\end{remark}

\medskip
The following immediate corollary is another instance of the general principle that groups whose second bounded cohomology is finite dimensional exhibit rigidity phenomena; compare for example with the case of actions by isometries on hermitian symmetric spaces (see \cite{B-I04}, \cite{B-I-W07a}).

\begin{corollary}\label{cor1.5}
Let $\Gamma$ be a countable group and assume that $H^2_b(\Gamma,\IR)$ is finite dimensional. Then there are, up to conjugation, only finitely many minimal strongly proximal $\Gamma$-actions on $S^1$.
\end{corollary}

Together with the information concerning centralizers of minimal actions we deduce from Theorem \ref{theo1.3},

\begin{corollary}\label{cor1.6} ~

\medskip\n
{\rm 1)} For a homomorphism $\rho: \Gamma \rightarrow \H(S^1)$, we have,
\begin{equation*}
\|\rho^* (e^b_\IR) \| \in \{0\} \cup\Big\{ \dis\frac{1}{2 k}: \;k \in\IN\Big\}
\end{equation*}

\n
and this value equals $(2k)^{-1}$ if and only if $\rho$ is quasi-conjugated to a minimal unbounded action whose centralizer is of order $k$.

\medskip\n
{\rm 2)} For minimal unbounded actions $\rho_1,\rho_2$ we have $\rho^*_1(e^b_\IR) = \rho^*_2(e^b_\IR)$ if and only if, up to conjugation,
\begin{equation*}
\rho_2(\gamma) = h(\gamma)\, \rho_1(\gamma), \;\;\gamma \in \Gamma
\end{equation*}

\medskip\n
where $h$ is a homomorphism with values in the centralizer of $\rho_1(\Gamma)$.
\end{corollary}

The extension criterium in Theorem \ref{theo1.2} leads to rigidity theorems when combined with the following two ingredients, namely vanishing theorems in bounded cohomology and the description of continuous homomorphisms from a locally compact group into $\H(S^1)$.

\medskip
Concerning the first ingredient, we know that the restriction map $H^2_{bc}(G, \IR) \rightarrow H^2_b(\Gamma, \IR)$ is a isomorphism in the following cases:

\medskip\n
1) Products (see \cite{B-M02}, \cite{Ka03}): $G = G_1 \times \dots \times G_n$ is a Cartesian product of locally compact second countable groups and $\Gamma$ has dense projection in every factor $G_i$.

\medskip\n
2) Higher rank Lie groups (see \cite{B-M02}): $G = \IG(k)$ where $\IG$ is a connected almost simple $k$-group, $k$ is a local field and ${\rm rank}_k\,\IG \ge 2$.

\medskip
Since it is elementary to classify continuous homomorphisms from a semi-simple Lie group over a local field into $\H(S^1)$ one obtains combining 1) and 2),

\begin{corollary}\label{cor1.7} {\rm (\cite{Gh99}, \cite{W-Z01})}

\medskip
Let $\Gamma < G : = \prod_{\alpha \in A} \IG_\alpha(k_\alpha)$ be an irreducible lattice where $k_\alpha$ are local fields, $\IG_\alpha$ is a connected, simply connected, almost simple $k_\alpha$-group of positive rank. Assume that the sum of the $k_\alpha$-ranks of $\IG_\alpha$ is at least $2$.

\medskip
For a homomorphism $\rho : \Gamma \rightarrow \H(S^1)$ one of the following holds up to quasi-conjugation:

\begin{itemize}
\item[{\rm 1)}] $\rho$ has a finite orbit.
\item[{\rm 2)}] $\rho$ is minimal unbounded and its strongly proximal factor $\rho_{sp}$ extends continuously to $G$ factoring via the projection on a factor of the form $\IG_\alpha(k_\alpha) = SL(2,\IR)$.
\end{itemize}
\end{corollary}

Concerning locally compact subgroups of $\H(S^1)$ one has, owing to the solution of Hilbert's fifth problem, a wealth of information and in particular those which are connected and minimal have a simple classification; they are up to conjugation, either Rot the subgroup of rotations or PSL$(2,\IR)_k$, the $k$-fold cyclic covering of PSL$(2,\IR)$ (see \cite{F01}, \cite{Gh01}). When studying continuous homomorphisms from a locally compact group $G$ into $\H(S^1)$ one has to deal with the fact that the image is not necessarily closed. In any case we have,

\begin{theorem}\label{theo1.8}
Let $G$ be a locally compact group and $\pi: G \rightarrow \H(S^1)$ a continuous and minimal action. Then one of the following holds:

\begin{itemize}
\item[{\rm 1)}] $\pi$ is conjugated into the group Rot of rotations and has dense image in it.
\item[{\rm 2)}] $\pi$ surjects onto PSL$(2, \IR)_k$ for some $k \ge 1$, up to conjugation.
\item[{\rm 3)}] ${\rm Ker} \,\pi$ is an open subgroup of $G$.
\end{itemize}
\end{theorem}

\n
From Theorem \ref{theo1.8} and the vanishing result for products mentioned above we obtain,

\begin{corollary}\label{cor1.9} {\rm (\cite{N05}, \cite{Ba-F-S06})}

\medskip
Let $G = G_1 \times \dots \times G_n$ be a product of locally compact second countable groups and $\Gamma < G$ a lattice with dense projections on each factor $G_i$. Assume that $\rho: \Gamma \rightarrow \H(S^1)$ is minimal unbounded. Then the strongly proximal quotient $\rho_{sp}$ extends continuously to $G$,
\begin{equation*}
(\rho_{sp})^{\rm ext} : G \rightarrow \H(S^1)
\end{equation*}
and we have one of the following, 
\begin{itemize}
\item[{\rm 1)}] ${\rm Ker} (\rho_{sp})^{\rm ext}$ is open in $G$.
\item[{\rm 2)}] Up to conjugation $(\rho_{sp})^{\rm ext}$ factors via a projection onto some factor $G_i$ followed by a continuous surjection onto PSL$(2, \IR) \subset \H(S^1)$.
\end{itemize}
\end{corollary}

Finally, as shown by Bader, Furman and Shaker, there is also in the context of actions on the circle a commensurator superrigidity theorem, which we state in a way which is somewhat different, but equivalent to their Thm.~C in \cite{Ba-F-S06}.

\begin{theorem}\label{theo1.10} {\rm (\cite{Ba-F-S06})}

\medskip
Let $G$ be a locally compact second countable group, $\Gamma < G$ a lattice and $\Lambda < G$ a subgroup such that $\Gamma \subset \Lambda \subset {\rm Comm}_G \,\Gamma$ and $\Lambda$ is dense in $G$. Let $\rho: \Lambda \rightarrow \H(S^1)$ be a homomorphism such that $\rho(\Gamma)$ is minimal and unbounded. Then the strongly proximal quotient $\rho_{sp}$ extends continuously to $G$, 
\begin{equation*}
(\rho_{sp})^{\rm ext} : G \rightarrow \H(S^1)
\end{equation*}
and we have one of the following,

\begin{itemize}
\item[{\rm 1)}] ${\rm Ker}(\rho_{sp})^{\rm ext}$ is open of infinite index in $G$.
\item[{\rm 2)}] $(\rho_{sp})^{\rm ext}$ surjects onto ${\rm PSL}(2, \IR) \subset \H(S^1)$.
\end{itemize}
\end{theorem}

Let us make the following comment about the hypothesis of Theorem \ref{theo1.10}. If $\rho: \Lambda \rightarrow \H(S^1)$ is a homomorphism such that $\rho(\Gamma)$ is unbounded then either $\rho(\Gamma)$ is minimal or there is an exceptional minimal set $K \subset S^1$ (see Section 3) which is easily seen to be $\rho(\Lambda)$-invariant. Thus $\rho$ is quasiconjugated to an action of $\Lambda$ for which $\Gamma$ is minimal and unbounded.

\medskip
The above result follows easily from the extension criterion (Theorem \ref{theo1.2}) and the following general fact concerning bounded cohomology:

\begin{theorem}\label{theo1.11} {\rm (\cite{I08})}

\medskip
Let $\Gamma < G$ and $\Gamma < \Lambda < {\rm Comm}_G \,\Gamma$ be as in Theorem {\rm \ref{theo1.10}}, in particular $\Lambda$ is dense in $G$. Then the image of the restriction map
\begin{equation*}
H^2_b (\Lambda, \IR) \rightarrow H^2_b (\Gamma, \IR)
\end{equation*}
coincides with the image of $H^2_{bc}(G, \IR)$ in $H_b^2(\Gamma,\IR)$.
\end{theorem}

We will leave to the interested reader the exercise of deducing Theorem \ref{theo1.10} from Theorem \ref{theo1.11} and refer to \cite{I08} for elementary proofs of the isomorphism results for products and higher rank groups mentioned above, as well as the proof of Theorem \ref{theo1.11}.

\bigskip\noindent
{\bf Acknowledgments:} Thanks to Luis Hernandez for his kind invitation to CIMAT where this work was completed. Thanks to U. Bader and A. Furman for sharing with me their result on boundary maps.

\section{Boundary maps}
\setcounter{equation}{0}

In this section we present a general existence and uniqueness result concerning measurable equivariant maps which is due to Bader-Furman and is of general interest in rigidity theory.

\medskip
Here and in the sequel $G$ is a second countable locally compact group, $G \times M \rightarrow M$ is a continuous action on a compact metrisable space $M$, $\mu \in \cM^1(G)$ is a spread out probability measure on $G$ and $(B,\nu_B)$ is a standard Lebesgue $G$-space such that $\nu_B$ is $\mu$-stationary.

\begin{theorem}\label{theo2.1} {\rm (\cite{Ba-F06})}

\medskip
Assume that 

\medskip\n
{\rm 1)} the $G$-action is minimal and strongly proximal.

\medskip\n
{\rm 2)} For every sequence $(g_n)_{n \ge 1}$ in $G$ there exists a subsequence $(n_k)_{k \ge 1}$ such that 
\begin{align*}
M  & \rightarrow  M
\\
m & \longmapsto g_{n_k} (m)
\end{align*}
converges pointwise.

\medskip
Then every measurable $G$-equivariant map 
\begin{equation*}
\varphi: B \rightarrow \cM^1 (M)
\end{equation*}
takes values in the subset of Dirac measures.
\end{theorem}

\begin{remark}\label{rem2.2} \rm
When $G$ is discrete countable and $M = S^1$, a result essentially equivalent to Theorem \ref{theo2.1} has been obtained by Deroin, Kleptsyn and Navas (see \cite{D-K-N07}).
\end{remark}

We recall here that a $G$-action is strongly proximal if for every $\mu \in \cM^1(M)$ the closure $\overline{G \mu} \subset \cM^1(M)$ contains a Dirac mass.

\begin{proof}
The main step consists in showing that the $G$-space $M$ is $\mu$-proximal in the sense of (2.6), Def.~3 in \cite{M91} VI: this means that for every $\mu$-stationary measure $\nu \in \cM^1(M)$ the map 
\begin{equation}\label{2.1}
\phi_\nu(\omega) : = \lim\limits_{n\rightarrow \infty} (x_1,\dots x_n) \,\nu, \;\;\omega = (x_n)_{n \ge 1}
\end{equation}

\n
with values in $\cM^1(M)$, defined almost everywhere on the infinite product $(S,\mu_S) : = \prod^\infty_{n = 1} \,(G,\mu)$, takes values in the set of Dirac measures; this we proceed to show now. The existence of the limit (\ref{2.1}) follows from the martingale convergence theorem (Prop. (2.4) in \cite{M91} VI); in addition one has the following remarkable fact (\cite{Fu63b}, \cite{G-R86} and Lemma 1.33 in \cite{F02}):
\begin{equation}\label{2.2}
\lim\limits_{n \rightarrow \infty} (x_1 \dots x_n\,g) \,\nu = \phi_\nu(\omega)
\end{equation}

\n
for $\mu_S$-almost every $\omega \in S$ and $\lambda$-almost every $g \in G$, where $\lambda$ is the left Haar measure on $G$. Fix now a sequence $\omega = (x_n)_{n \ge 1}$ and a subset $E \subset G$ of full $\lambda$-measure such that (\ref{2.2}) holds for $\omega$ and all $g \in E$. Passing to a subsequence we may assume that the limit 
\begin{equation*}
s(m): = \lim\limits_{n \rightarrow \infty} (x_1 \dots x_n) (m)
\end{equation*}

\n
exists for every $m \in M$. Then the map $s: M \rightarrow M$, although in general not continuous, admits points of continuity; in fact they form a $G_\delta$-subset of $M$. Let then $m \in M$ be a point of continuity of $s$. Since the $G$-action is minimal and strongly proximal, there exists a sequence $(g_n)_{n \ge 1}$ in $G$ such that $\lim_{n \rightarrow \infty} g_n \nu = \delta_m$. By the continuity of the $G$-action on $M$ there exists for each $n \ge 1$ an open neighborhood $U_n \ni g_n$ such that whenever $(h_n)_{n \ge 1} \in \prod_{n \ge 1} U_n$, $ \lim_{n \rightarrow \infty}  \,h_n \nu = \delta_m$. Since $\lambda (U_n) > 0$, $\forall n \ge 1$, we may choose $h_n \in U_n \cap E$ and obtain a sequence $(h_n)_{n \ge 1}$ such that
\begin{equation}\label{2.3}
\begin{array}{ll}
{\rm a)} & \lim\limits_{n \rightarrow \infty}\;h_n \,\nu = \delta_m
\\[2ex]
{\rm b)} & \lim\limits_{n \rightarrow \infty} \;(x_1 \dots x_n\,h_k) \, \nu = \phi_\nu(\omega), \;\mbox{for every $k \ge 1$} \,.
\end{array}
\end{equation}

\n
Let $V$ be a neighborhood of $s(m)$ and $W \ni m$ an open neighborhood such that $W \subset s^{-1}(V)$. Taking into account that $s : M \rightarrow M$ is a Borel map we have,
\begin{equation*}
(sh_k) \nu (V) = (h_k \,\nu) \big(s^{-1}(V)\big) \ge h_k \,\nu (W)
\end{equation*}

\n
which together with $\lim\limits_{k \rightarrow \infty} h_k \, \nu(W) = 1$ (see (\ref{2.3}) a)) implies that
\begin{equation}\label{2.4}
\lim\limits_{k \rightarrow \infty} \;(sh_k) (\nu) = \delta_{s(m)}\,.
\end{equation}
For every $k \ge 1$ and every continuous function $f \in C(M)$ we have
\begin{equation*}
\begin{split}
(sh_k) \nu (f) & = \; \dis\int_M f(sh_k \,x) \,d\nu(x)
\\
& = \;\dis\int_M \;\lim\limits_{n \rightarrow \infty} \;f(x_1 \dots x_n\,h_k \,x) \,d\nu(x)
\\[1ex]
& = \; \phi_\nu(\omega)(f)
\end{split}
\end{equation*}

\n
where in the second equality we have used the definition of $s$ and the continuity of $f$, while in the third we have used dominated convergence and (\ref{2.3}) b).

\medskip
Thus $sh_k \,\nu = \phi_\nu(\omega)$ for every $k \ge 1$ which together with (\ref{2.4}) implies that
\begin{equation*}
\phi_\nu(\omega) = \delta_{s(m)}\,,
\end{equation*}

\n
and hence since $\omega$ could be chosen from a subset of full measure, $\phi_\nu$ takes values in the set of Dirac masses almost everywhere; thus $M$ is $\mu$-proximal. Let now $(B, \nu_B)$ and $\varphi: B \rightarrow \cM^1(M)$ be as in Theorem \ref{theo2.1}. Then $\varphi_*(\nu_B) \in \cM^1\big(\cM^1(M)\big)$ is a $\mu$-stationary probability measure on $\cM^1(M)$ and since $M$ is $\mu$-proximal, Prop.~(2.9) in \cite{M91} VI implies that the support of $\varphi_*(\nu_B)$ is contained in the subset of $\cM^1(M)$ consisting of Dirac masses. This concludes the proof of Theorem \ref{theo2.1}
\end{proof}

In the following proposition we give examples of groups of homeomorphisms with the property in Theorem \ref{theo2.1} (2).

\begin{proposition}\label{prop2.2}
Let $(G,M)$ be one of the following pairs consisting of a compact metric space $M$ and a subgroup $G$ of the group of homeomorphisms of $M$:

\begin{itemize}
\item[{\rm 1)}] $M = \IP V_k$, $G = {\rm PGL}(V_k)$ where $V_k$ is a finite dimensional vector space over a local field $k$.

\item[{\rm 2)}] $M = S^1$, $G = \H (S^1)$.
\end{itemize}

\noindent
Then for any sequence $(g_n)_{n \ge 1}$ there is a subsequence such that $(g_{n_k})_{k \ge 1}$ converges pointwise.
\end{proposition}

\n
{\bf Question 2.3.}

\medskip\n
Let $X$ be a proper complete CAT$(0)$-space, $M = X(\infty)$ its visual boundary and $G = {\rm Isom}(X)$. Does $(G, M)$ satisfy the conclusion of Proposition \ref{prop2.2}?

\medskip
The first statement in Proposition \ref{prop2.2} follows easily by recurrence on $\dim V_k$ using Furstenberg's lemma (\cite{Fu63a}). Concerning the second, we may compose $g_n$ by an appropriate rotation $r_n$ such that $r_n \circ g_n$ fixes a point in $S^1$ and, passing to a subsequence we may assume that $(r_n)_{n \ge 1}$ converges uniformly; Proposition \ref{prop2.2} (2) will then follow from

\begin{lemma}\label{lem2.4}
Let $f_n: [0,1] \rightarrow [0,1]$, $n \ge 1$, be a sequence of monotone increasing maps. Then there exists a subsequence converging pointwise.
\end{lemma}

\begin{proof}
We may assume that $f(r) : = \lim_{n \rightarrow \infty} f_n(r)$ exists $\forall r \in \IQ \cap [0,1]$. Then $f$ is monotone increasing on $\IQ \cap [0,1]$ and thus 
\begin{equation*}
F(x) : = \sup\{f (r): r \le x, \;r \in \IQ \cap [0,1]\}
\end{equation*}

\n
is a well defined function on $[0,1]$ which is monotone increasing and extends $f$. The set $\cD \subset [0,1]$ of discontinuities of $F$ is at most countable; passing to a subsequence we may assume then that $\lim_{n \rightarrow \infty} f_n(d)$ exists for every $d \in \cD$. 

\medskip
Let now $t \notin \cD$ be a point of continuity of $F$. For every $t^\prime, t^{\prime\prime} \in \IQ$ with $t^\prime \le t \le t^{\prime\prime}$ we have $F(t^\prime) \le \xi \le F(t^{\prime\prime})$ where $\xi$ is any accumulation point of the sequence $\big(f_n(t)\big)_{n \ge 1}$. Letting $t^\prime, t^{\prime\prime}$ approach $t$ and using the continuity of $F$ we get $\xi = F(t)$ and hence $\lim_{n \rightarrow \infty} f_n(t) = F(t)$. 
\end{proof}

\section{Elementary properties of actions on the circle}
\setcounter{equation}{0}

In this section we establish the basic dichotomy concerning actions on the circle in terms of their bounded real Euler class, announced in the introduction, and show the uniqueness of the strongly proximal quotient. Along the way we'll take the opportunity to recall some well known facts which will be used later on.

\medskip
Let $S^1 = \IZ \backslash \IR$ and consider the central extension
\begin{equation*}
0 \rightarrow \IZ \stackrel{i}{\rightarrow} {\rm Homeo}^+_\IZ (\IR) \stackrel{p}{\rightarrow} \H(S^1) \rightarrow (e)
\end{equation*}

\n
where ${\rm Homeo}^+_\IZ (\IR)$ is the group of increasing homeomorphisms of $\IR$ commuting with integer translations and $i( n) = T^n$, where $T(x)  = x+1, x \in \IR$. A section of $p$ is obtained by associating to every $f \in \H(S^1)$ the unique lift $\overline{f}: \IR \r \IR$ satisfying $\overline{f}(0) \in [0,1)$. Then we have
\begin{equation}\label{3.1}
\overline{f g} \;T^{c(f,g)} = \overline{f} \,\overline{g}, \;f,g \in \H(S^1)
\end{equation}

\n
and $c(f,g) \in\{0,1\}$ is an inhomogeneous $2$-cocycle called the Euler cocycle whose class $e \in H^2 (\H(S^1), \IZ)$ is called the Euler class. Considering $c$ as a $\IZ$-valued resp. $\IR$-valued bounded cocycle leads respectively to the bounded Euler class $e^b \in H^2_b (\H(S^1), \IZ)$ and the bounded real Euler class $e^b_\IR \in H^2_b (\H(S^1), \IR)$. Given a group $\Gamma$ and a homomorphism $\rho : \Gamma \r \H(S^1)$ we obtain accordingly three invariants associated to $\rho$ namely, the Euler class $\rho^*(e) \in H^2(\Gamma,\IZ)$, the bounded Euler class $\rho^*(e^b) \in H_b^2(\Gamma, \IZ)$ and the real bounded Euler class $\rho^*(e^b_\IR) \in H_b^2(\Gamma,\IR)$. The information contained in $\rho^*(e)$ is of purely algebraic nature; $\rho^*(e) = 0$ if and only if $\rho$ lifts to a homomorphism with values in ${\rm Homeo}^+_{\IZ}(\IR)$, which for example is the case when $\Gamma$ is a free group. Concerning the bounded Euler class we have on the other hand

\begin{theorem}\label{theo3.1} {\rm (\cite{Gh87})}

\medskip
Two actions $\rho_0, \rho_1$ are quasi-conjugated if and only if $\rho^*_0 (e^b) = \rho_1^*(e^b)$.
\end{theorem}

There are several equivalent ways to look at quasi-conjugations. Ghys' original definition is as follows: a quasi-conjugation is a map $h: S^1 \r S^1$ induced by a monotone increasing map $\overline{h}: \IR \r \IR $ commuting with integer translations; observe that a quasi-conjugation is not necessarily continuous. Two actions $\rho_0,\rho_1$ are then quasi-conjugated if there is such a map $h$ with 
\begin{equation}\label{3.2}
h \rho_0(\gamma) = \rho_1(\gamma) \,h, \;\;\forall \gamma \in \Gamma\,.
\end{equation}
Quasi-conjugation is an equivalence relation.

\medskip
Observe that if we set $(\ov{h})_+ (x) = \lim_{y > x} \,\ov{h}(y)$, $(\ov{h})_-\, (x) = \lim_{y < x} \,\ov{h}(y)$, then $(\ov{h})_+$ and $(\ov{h})_-$ induce respectively a right continuous $h_+$ and a left continuous $h_-$ quasi-conjugacy and if (\ref{3.2}) holds then we have $h_{\pm} \,\rho_0(\gamma) = \rho_1(\gamma) \,h_\pm$, $\forall \gamma \in \Gamma$ as well.

\medskip
An intrinsic way to define quasi-conjugation rests on the consideration of the orientation cocycle $o: (S^1)^3 \r \{ -1,0,1\}$ defined as follows:
\begin{equation}\label{3.3}
o(x,y,z) = \left\{ \begin{array}{rl}
1 & \mbox{if $x,y,z$ are positively oriented,}
\\[1ex]
-1 & \mbox{if $x,y,z$ are negatively oriented,}
\\[1ex]
0 & \mbox{if $x,y,z$ are not pairwise distinct.}
\end{array}\right.
\end{equation}

\n
Then a map $h : S^1 \r S^1$ is a quasi-conjugacy if and only if
\begin{equation*}
o\big(h(x), h(y), h(z)\big) \ge 0 \;\;\mbox{whenever} \;\;o(x,y,z) \ge 0\,.
\end{equation*}

\n
A third characterization is obtained as follows. Define for $x,y \in S^1$ the open interval $(x,y) : = \{z \in S^1: o (x,z,y) = 1\}$ and $[x,y]$, $(x,y]$, $[x,y)$ in the obvious way. Given a point $b \in S^1$ and a probability measure $\mu \in \cM^1(S^1)$, the map
\begin{equation}\label{3.4}
h_{b,\mu}\big(x)  : = \mu([b,x)\big) \;{\rm mod} \, \IZ 
\end{equation}

\n
is a left continuous quasi-conjugacy and every such quasi-conjugacy is obtained this way for a unique pair $(b,\mu)$. The behaviour under composition with homeomorphisms is particularly simple,
\begin{equation}\label{3.5}
h_{b,\mu}\big(\varphi^{-1}(x)\big) = h_{\varphi(b), \varphi_*(\mu)}(x), \;\;\varphi \in \H(S^1)\,.
\end{equation}

\n
A very useful tool in the study of group actions on $S^1$ is the following trichotomy for $\rho: \Gamma \r \H(S^1)$, see (\cite{He-Hi83}, Chap. IV-3)
\begin{equation}\label{3.6}
\begin{array}{ll}
{\rm 1)} & \mbox{There is a finite orbit and all finite orbits have the same cardinality.}
\\[1ex]
{\rm 2)} & \mbox{The action is minimal.}
\\[1ex]
{\rm 3)} & \mbox{There is a unique proper minimal invariant closed subset $F \underset{+}{\subset} S^1$,}
\\[-1ex]
&\mbox{which is a Cantor set.}
\end{array}
\end{equation}

\n
In the first case $\rho$ is quasi-conjugated to an action $\rho_0$ by rotations; a quasi-conjugacy is given by $h_{b,\mu}$ where $\mu$ is the uniform measure along a finite orbit. Then $\rho_0(\Gamma)$ is finite cyclic.

\medskip
In the second case $\rho(\Gamma)$ is either bounded (equicontinuous), preserves hence a probability measure $\mu$ necessarily with full support and no atoms, and can be conjugated into the subgroup of rotations by the homeomorphism $h_{b,\mu}$; or $\rho(\Gamma)$ is unbounded. 

\medskip
In the third case one can collapse each connected component of $S^1 \backslash F$ to a point and obtain this way a continuous quasi-conjugacy between $\rho$ and a minimal action. In fact such a quasi-conjugacy is given by $h_{b,\mu}$ where $\mu$ is a probability measure without atoms and with support precisely the Cantor set $F$.

\medskip
With this at hand, and denoting as usual
\begin{equation*}
{\rm rot}: \;\;\H(S^1) \r \IZ \backslash \IR
\end{equation*}
the rotation number function we have:
\begin{proposition}\label{prop3.2}
Let $\rho : \Gamma \r \H(S^1)$ be a homomorphism. 

\begin{itemize}
\item[{\rm 1)}] $\rho^*(e^b_\IR) = 0$ if and only if $\rho$ is quasi-conjugated to a subgroup of the group of rotations. When this is the case,
\begin{align*}
\Gamma & \r \;\IZ \backslash \IR
\\
\gamma &\longmapsto \; {\rm rot}\, \rho(\gamma)
\end{align*}
is a homomorphism and $\rho$ is quasi-conjugated to the action by rotations defined by it.

\item[{\rm 2)}]  $\rho^*(e^b_\IR) \not= 0$ if and only if $\rho$ is quasi-conjugated to a minimal unbounded action.
\end{itemize}
\end{proposition}

\begin{proof} 1) Assume $\rho^*(e^b_\IR) = 0$. Then there exists $\beta: \Gamma \r \IR$ bounded with
\begin{equation*}
d\beta (\gamma, \eta): = \beta (\gamma) - \beta(\gamma \eta) + \beta(\eta) = c\big(\rho(\gamma), s(\eta)\big), \quad \forall \gamma, \eta \in \Gamma\,.
\end{equation*}

\n
Since $c(~,~)$ takes values in $\IZ$, the map $f: \Gamma \r \IZ \backslash \IR$, $\gamma \longmapsto f(\gamma) = \beta(\gamma) \;{\rm mod} \,\IZ$, is a homomorphism. Let $\rho_0(\gamma) (x) : = x + f(\gamma)$, $x \in \IZ \backslash \IR$ be the corresponding action by rotations. Denoting by $\{~\}$ the fractional part of a real number, we have then $\overline{\rho_0(\gamma)} \,(x) = x + \{\beta(x)\}$, $\forall \gamma \in \Gamma$ and hence $c\big(\rho_0(\gamma), \rho_0(\eta)\big) = d(\{\beta\}) (\gamma, \eta)$, $\forall \gamma, \eta \in \Gamma$. This implies for all $\gamma, \eta \in \Gamma$:
\begin{equation*}
c(\rho_0\big(\gamma), \rho_0(\eta)\big) + d ([\beta]) (\gamma, \eta) = c\big(\rho(\gamma), \rho(\eta)\big)
\end{equation*}
where $[\beta]: \Gamma \r \IZ$ is the bounded function associating to $\gamma$ the integral part $[\beta(\gamma)]$ of $\beta(\gamma)$. As a result $\rho_0^*(e^b) = \rho^*(e^b)$ and by Ghys' theorem (Thm 3.1) $\rho$ and $\rho_0$ are quasi-conjugated. Since the rotation number is an invariant of quasi-conjugation we have $f(\gamma) = {\rm rot} \, \rho(\gamma)$, $\forall \gamma \in \Gamma$ and hence $\gamma \longmapsto {\rm rot} \, \rho(\gamma)$ is a homomorphism.

\medskip
Conversely, if $\rho$ is quasi-conjugated to a homomorphism $\rho_0: \Gamma \r {\rm Rot}$ then $\rho^*(e^b) = \rho^*_0(e^b)$ and hence $\rho^*(e^b_\IR) = \rho^*_0(e^b_\IR)$. But since Rot is compact, $e^b_\IR |_{\rm Rot} = 0$ which implies $0 = \rho_0^*(e^b_\IR) = \rho^*(e^b_\IR)$. 

\medskip\n
2) If $\rho^*(e^b_\IR) \not= 0$ then it follows from 1) that $\rho$ is not quasi-conjugated into Rot and hence by the trichotomy (\ref{3.6}) and the discussion following it, $\rho$ is quasi-conjugated to a minimal unbounded action.
\end{proof}

Since bounded cohomology with $\IR$-coefficients vanishes for amenable groups we conclude that,
\begin{corollary}\label{cor3.3}
If $\Gamma$ is amenable then for any homomorphism $\rho : \Gamma \r \H(S^1)$, the map
\begin{align*}
\Gamma & \r \;\IZ \backslash \IR
\\
\gamma &\longmapsto \; {\rm rot}\, \rho(\gamma)
\end{align*}
is a homomorphism.
\end{corollary}

A simpler proof of the last corollary consists in observing that the rotation number with respect to a $\Gamma$-invariant measure coincides with Poincar\'e's rotation number and is clearly a group homomorphism.

Now we turn to a closer study of minimal unbounded actions. In this context the following definition will be useful:

\begin{definition}\label{def3.4}
A non-trivial $S^1$-factor of an action $\rho: \Gamma \r \H(S^1)$ is a pair $(\varphi, \rho_0)$ consisting of a non-constant continuous map $\varphi: S^1 \r S^1$ and an action $\rho_0: \Gamma \r \H(S^1)$ with $\varphi \rho(\gamma) = \rho_0 (\gamma) \varphi$, $\forall \gamma \in \Gamma$.
\end{definition}

Denoting by $\cZ_{\cH^+}(L)$ the centralizer in $\H(S^1)$ of a subgroup $L$ we have:

\begin{theorem}\label{theo3.5} {\rm (\cite{M00}, \cite{Gh01})}

\medskip
Let $\rho : \Gamma \r \H(S^1)$ be minimal unbounded. Then the centralizer $C_{\rho}: =\cZ_{\cH^+}\big(\rho(\Gamma)\big)$ of $\rho(\Gamma)$ in $\H(S^1)$ is finite cyclic and the $S^1$-factor $(\varphi,\rho_0)$, where $\varphi: S^1 \r C_\rho \backslash S^1$ is the quotient map and $\rho_0$ the quotient action, is strongly proximal.
\end{theorem}

For minimal circle actions one observes that strong proximality is equivalent to the property that every proper interval can be contracted; we recall that a subset $F \subset S^1$ can be contracted if there exists a sequence $(\gamma_n)_{n \ge 1}$ such that $\lim_{n \r \infty} \,{\rm diam}(\rho (\gamma_n) F) = 0$. The idea of the proof of Theorem \ref{theo3.5} is then as follows: since $\rho$ is minimal unbounded, in particular not equicontinuous, every point $x \in S^1$ has a neighborhood which can be contracted. For every $x \in S^1$ the set $\{I = [x,y) | I$ can be contracted$\}$ is then totally ordered by inclusion, non-empty, and hence contains a unique maximal element which we denote $[x,\theta(x))$. Then one verifies that $\theta: S^1 \r S^1$ is a homeomorphism which is periodic and $C_\rho = \IZ_{\cH^+}\big(\rho(\Gamma)\big) = \langle \theta \rangle$. By construction, every proper interval in the quotient $C_\rho \backslash S^1$ can be contracted. We record the following immediate,

\begin{corollary}\label{cor3.6}
Let $\rho: \Gamma \r \H(S^1)$ be minimal unbounded. Then $\rho$ is strongly proximal if and only if $\cZ_{\cH^+}\big(\rho(\Gamma)\big) = (e)$.
\end{corollary}

We now establish the uniqueness of a proximal $S^1$-factor:

\begin{proposition}\label{prop3.7}
Let $\rho : \Gamma \r \H(S^1)$ be minimal unbounded and let $(\varphi, \rho_0)$ be a non-trivial $S^1$-factor. If $\rho_0$ is strongly proximal then $\varphi$ is up to conjugation the quotient map by $\cZ_{\cH^+}\big(\rho(\Gamma)\big)$.
\end{proposition}

First we'll need
\begin{lemma}\label{lem3.8}
If $\rho$ is minimal and $(\rho_0,\varphi)$ is a non-trivial $S^1$-factor then deg $\varphi \not= 0$ and $\varphi$ is surjective.
\end{lemma}

\begin{proof}
Assume that deg $\varphi = 0$ and let $\wt{\varphi}: S^1 \r \IR$ be a lift of $\varphi$. Then there is for every $\gamma \in \Gamma$ a $c(\gamma) \in \IZ$ such that
\begin{equation*}
\wt{\varphi} \rho(\gamma) = T^{c(\gamma)} \;\overline{\rho_0(\gamma)} \;\wt{\varphi}\,.
\end{equation*}

\n
Let $M = \sup\limits_{x \in S^1} \,\wt{\varphi}(x)$ and $x_0 \in S^1$ with $M = \wt{\varphi}(x_0)$. Then

\vspace{-2ex}
\begin{equation*}
\begin{split}
\wt{\varphi} (x_0) & = \sup\limits_{x} \;\wt{\varphi}(x) = \sup\limits_x \;\wt{\varphi} (\rho(\gamma) x)
\\
& = c(\gamma) + \overline{\rho_0(\gamma)} \;\big(\sup\limits_x \;\wt{\varphi}(x)\big)
\\
& = c(\gamma) + \overline{\rho_0(\gamma)} \;\big(\wt{\varphi}(x_0)\big)
\end{split}
\end{equation*}

\n
which implies that $\varphi(x_0) \in S^1$ is fixed by $\rho_0(\gamma)$, $\forall \gamma \in \Gamma$. The non-empty closed subset $\varphi^{-1} \big(\varphi(x_0)\big)$ being $\rho(\Gamma)$-invariant equals $S^1$ and thus $\varphi$ is constant.
\end{proof}

{\it Proof of Proposition \ref{prop3.7}:} Let $\theta$ be a generator of $\cZ_{\cH^+}\big(\rho(\Gamma)\big)$ constructed as indicated in the discussion following Theorem \ref{theo3.5}. We claim that $\varphi \theta = \varphi$. For $x \in S^1$ we consider the closed connected subset $I_x : = \varphi([x,\theta(x)])$ and, for a fixed $x$, distinguish two cases: 

\medskip\n
a) $I_x = S^1$: then, either $\varphi\big(\theta(x)\big) = \varphi(x)$ which implies by equivariance that $\varphi \theta$ and $\varphi$ coincide on $\rho(\Gamma) \cdot x$ and hence everywhere; or there is $y \in \big(x, \theta(x)\big)$ with $\varphi(x) = \varphi(y)$ and $\varphi([x,y]) = S^1$. But by the construction of $\theta$, $[x,y]$ can be contracted which by continuity of $\varphi$ implies the same for $\varphi([x,y]) = S^1$; this is absurd. 

\medskip\n
b) $I_x$ is a proper subset of $S^1$: observe that for all $\gamma \in \Gamma$, $\rho_0(\gamma) I_x = I_{\rho(\gamma) x}$. Since $I_x$ is a proper interval and $\rho_0$ is strongly proximal we can choose a sequence $(\gamma_n)_{n \ge 1}$ such that diam $(\rho_0(\gamma_n) I_x)$ tends to zero for $n \r \infty$ and $(\rho(\gamma_n)x)_{n \ge 1}$ converges to some point $y$, which by continuity of $\varphi$ implies that $\varphi(y) = \varphi \theta (y)$. Again by using equivariance and minimality we get $\varphi = \varphi \theta$.

\medskip
Thus the claim is established and therefore by passing to the quotient by \linebreak $\cZ_{\cH^+}\big(\rho(\Gamma)\big)$ we may assume that $\rho$ is strongly proximal as well. We already know (Lemma \ref{lem3.8}) that deg $\varphi \not= 0$ and $\varphi$ is surjective. Assume then that there exists $s \not= t$ in $S^1$ with $\varphi(s) = \varphi(t)$. Since deg $\varphi \not= 0$ we have either $\varphi([s,t]) = S^1$ or $\varphi([t,s]) = S^1$. Assume without loss of generality that the former occurs. Since $[s,t]$ is a proper intervall in $S^1$ it can be contracted by $\rho$ which clearly implies a contradiction. Thus $\varphi$ is injective as well and hence a homeomorphism. \hfill $\square$

\section{Strongly proximal actions, boundary maps and cocycles}
\setcounter{equation}{0}

In this section we study continuous group actions of a locally compact group $G$ on the circle using boundary maps. The main goal is Theorem \ref{theo4.5} which gives a description of the space of conjugacy classes of minimal strongly proximal actions in terms of an explicit space of cocycles on appropriate Poisson boundaries of $G$.

\subsection{Cohomological preliminaries}

For a non-discrete group some care must be taken when defining the analogues of the various Euler classes in the continuous context. Let $G$ be locally compact, second countable and let $A$ be either $\IZ$ or $\IR$. Then $H^\bullet_{bc}(G,A)$ denotes the cohomology defined via $A$-valued bounded Borel cochains on $G$; in the case $A = \IR$ these cohomology groups coincide with those obtained by taking the subcomplex of bounded continuous cochains on $G$; for more details see \cite{B-I-W06}, \S 2.3. Given a continuous action $\rho: G \r \H(S^1)$ the function $(g,h) \longmapsto c\big(\rho(g), \rho(h)\big)$ is a bounded Borel (inhomogeneous) $2$-cocycle on $G$ and leads to classes $\rho^*(e^b) \in H^2_{bc} (G, \IZ)$ and $\rho^*(e^b_\IR) \in H^2_{bc} (G, \IR)$ which are respectively the bounded Euler class and the real bounded Euler class of the continuous action.

\medskip
Recall that if $(B,\nu_B)$ is a standard Lebesgue $G$-space which is doubly ergodic and amenable there is a canonical isometric isomorphism, see \cite{Mo01} Thm. 7.5.3,
\begin{equation}\label{4.1}
H_{bc}^2(G, \IR) \simeq Z \,L_{\rm alt}^\infty(B^3,\IR)^G
\end{equation}

\noindent
where the right-hand side is the space of measurable, essentially bounded, alternating $G$-invariant cocycles on $B^3$.

\begin{proposition}\label{prop4.1}
Let $\rho : G \r \H(S^1)$ be a continuous action and assume that there exists a measurable $G$-equivariant map $\varphi : B \rightarrow S^1$. Then under the isomorphism in {\rm (\ref{4.1})} the class $\rho^*(e^b_\IR)$ corresponds to the cocycle 
\begin{align*}
B^3 &  \rightarrow \IR
\\[-1ex]
(x,y,z) & \longmapsto - \mbox{\footnotesize $\dis\frac{1}{2}$} \;o\big(\varphi(x), \varphi(y), \varphi(z)\big)\,.
\end{align*}
If in addition the essential image of $\varphi$ contains at least three points, we have
\begin{equation*}
\|\rho^* (e^b_\IR)\| = \mbox{\footnotesize $\dis\frac{1}{2}$} \,.
\end{equation*}
\end{proposition}

We will need the following explicit relationship between the Euler and orientation cocycle; see for instance \cite{I02} Lemma 2.1.

\begin{lemma}\label{lem4.2}
For every $f,g \in \H(S^1)$ we have 
\begin{align*}
2c(f,g) = & - o\big(\dot{o},f(\dot{o}), \,fg(\dot{o})\big) + 1
\\
& + \;\big(\delta_{\dot{o}} (fg(\dot{o})\big) - \delta_{\dot{o}}\big(f(\dot{o})\big) - \delta_{\dot{o}} \big(g(\dot{o})\big)\big)\,.
\end{align*}
where $\dot{o} \in \IZ \backslash \IR$.
\end{lemma}

\n
{\it Proof of Proposition \ref{prop4.1}.} This is a standard argument using \cite{B-I02}. Since $G$ acts continuously on $S^1$, the $G$-complex $(B^\infty_{\rm alt}\big((S^1)^{\bullet + 1}, \IR), d)$ of bounded alternating Borel cochains on $S^1$ is a strong resolution of the trivial $G$-module $\IR$ (see \cite{B-I02} Prop.~2.1). Then, owing to the properness of the $G$-action on $G$ and the amenability of the $G$-action on $B$, the two complexes of $G$-modules $(L_{\rm alt}^\infty(G^{\bullet + 1}, \IR),d)$ and $(L_{\rm alt}^\infty(B^{\bullet + 1}, \IR),d)$ are strong resolutions of $\IR$ by relatively injective $G$-Banach modules \cite{Mo01} \S 7.5. Then there is a morphism of $G$-complexes
\begin{equation}\label{4.2}
c^\bullet : L_{\rm alt}^\infty (G^{\bullet + 1}, \IR) \r L_{\rm alt}^\infty (B^{\bullet + 1}, \IR)
\end{equation}

\n
extending the identity and any two such are $G$-equivariantly homotopic; the canonical map induced in cohomology is then an isometric isomorphism which in degree two gives the one mentioned in (\ref{4.1}).

\medskip
Consider the morphisms of $G$-resolutions 
\begin{equation*}
\rho^{(n)} : B^\infty_{\rm alt} \big(S^1)^{n + 1}, \IR\big) \r L_{\rm alt}^\infty (G^{n+ 1}, \IR)
\end{equation*}
and
\begin{equation*}
\varphi^{(n)} : B^\infty_{\rm alt} \big((S^1)^{n + 1}, \IR\big) \r L_{\rm alt}^\infty (B^{n+ 1}, \IR)
\end{equation*}

\medskip\n
defined respectively for $\alpha : (S^1)^{n+1} \r \IR$ by 
\begin{equation*}
\rho^{(n)} (\alpha) (g_0, \dots,g_n) : = \alpha (\rho(g_0) \dot{o}, \dots, \rho(g_n) \,\dot{o})
\end{equation*}
and
\begin{equation*}
\varphi^{(n)} (\alpha) (x_0, \dots,x_n) : = \alpha \big(\varphi(x_0) , \dots, \varphi(x_n)\big)\,.
\end{equation*}

\medskip\n
Then $c^{(\bullet)} \,\rho^{(\bullet)}$ and $\varphi^{(\bullet)}$ are both morphisms of $G$-complexes extending the identity and therefore, since $(L_{\rm alt}^\infty (B^{\bullet + 1}, \IR),d)$ is a resolution by relatively injective $G$-modules, induce the same map in cohomology. Now we specialize this to degree $n=2$. Let $[o] \in H^2 \big(B^\infty_{\rm alt}(S^{\bullet + 1})\big)$ be the class defined by the orientation cocycle. Then we have by Lemma \ref{lem4.2}: $\rho^{(2)}([o]) = 2 \rho^*(e^b_\IR)$ and in addition that under the isometric isomorphism $c^{(2)}$, $\rho^{(2)}([o])$ is represented by $(x,y,z) \longmapsto o\big(\varphi(x),\varphi(y),\varphi(z)\big)$. In particular $2\|\rho^*(e^b_\IR)\|$ is the essential supremum of the latter cocycle. \hfill $\square$

\medskip
We will see below that the existence of a map as in Proposition \ref{prop4.1} imposes strong conditions on the type of action considered. Before we turn to this we'll need,
\begin{lemma}\label{lem4.3}
Let $\rho_0,\rho_1: G \r \H(S^1)$ be continuous actions and $\varphi: S^1 \r S^1$ a degree $k$ covering such that
\begin{equation*}
\rho_0(g) \,\varphi = \varphi \,\rho_1(g), \quad \forall g \in G\,.
\end{equation*}

\n
Then we have the following equality
\begin{equation*}
\rho_0^*(e^b) = k  \, \rho_1^*(e^b)
\end{equation*}
in $H^2_{bc}(G, \IZ)$.
\end{lemma}

\begin{proof}
Let $p : \IR \r \IZ \backslash \IR$ denote projection and $M_k : \IR \r \IR$, $x \longmapsto k \cdot x$, so that up to conjugation we have $p\,M_k = \varphi p$. Then we have $\forall g \in G$:
\begin{equation*}
p\,M_k \;\overline{\rho_1(g)} = \varphi\,p \;\overline{\rho_1(g)} = \varphi\,\rho_1(g) \,p = \rho_0(g) \,\varphi p = \rho_0 (g) \,p \,M_k = p \;\overline{\rho_0(g)} \;M_k\,.
\end{equation*}

\n
Therefore there is $\alpha(g) \in \IZ$ such that
\begin{equation}\label{4.3}
M_k \;\overline{\rho_1(g)} = T^{\alpha(g)}  \;\overline{\rho_0(g)} \;M_k\,.
\end{equation}

\n
Evaluation at $0 \in \IR$ gives
\begin{equation*}
k \;\overline{\rho_1(g)}\;(0) = \alpha(g) + \;\overline{\rho_0(g)} \;(0)
\end{equation*}

\n
which first shows that $\alpha : G \r \IZ$ is a Borel function and furthermore implies that $\sup_{g \in G} \;|\alpha (g)| \le k + 1$.

\medskip
Applying (\ref{4.3}) to products and using the relation defining the Euler cocycle we get
\begin{equation*}
k \,c\big(\rho_1(g), \rho_1(h)\big) + \alpha (gh) - \alpha(g) - \alpha(h) = c\big(\rho_0 (g), \rho_0(h)\big)
\end{equation*}
which proves the lemma.
\end{proof}

Now we come to our first application.

\begin{corollary}\label{cor4.4}
Let $\mu \in \cM^1(G)$ be a spread out probability measure on $G$ and $(B,\nu_B)$ a standard Lebesgue $G$-space, where $\nu_B$ is $\mu$-stationary. Assume that $B$ is doubly ergodic and amenable. Given a minimal unbounded action $\rho : G \r \H(S^1)$ the following are equivalent:
\begin{itemize}
\item[{\rm (1)}] $\rho$ is strongly proximal.
\item[{\rm (2)}] There exists a $G$-equivariant measurable map $\varphi : B \r S^1$.
\item[{\rm (3)}] $\|\rho^*(e^b_\IR)\| = \frac{1}{2}$.
\end{itemize}
\end{corollary}

\begin{proof}
(1) $\Longrightarrow$ (2): follows immediately from Theorem \ref{theo2.1} and Proposition \ref{prop2.2} (3). 

\medskip\n
(2) $\Longrightarrow$ (3): follows from Proposition \ref{prop4.1}.

\medskip\n
(3) $\Longrightarrow$ (1): Let $(\psi,\rho_0)$ be the strongly proximal factor of $\rho$. Applying the implication (1) $\Longrightarrow$ (3) to $\rho_0$ we get $\|\rho_0^*(e^b_\IR)\| = \frac{1}{2}$, and for $k = {\rm deg} \,\psi$, we get from Lemma \ref{lem4.3} that $\rho_0^*(e^b_\IR) = k \cdot \rho^*(e^b_\IR)$ which implies $k = 1$ and hence that $\rho$ is strongly proximal.
\end{proof}

\subsection{Minimal strongly proximal actions and cocycles}

The main result of this section is
\begin{theorem}\label{theo4.5}
Let $G$ be locally compact second countable and let $\mu \in \cM^1(G)$, $(B,\nu_B)$ be as in Corollary {\rm \ref{cor4.4}}. Then there is a bijection between

\medskip\n
{\rm 1)} The set of conjugacy classes of continuous $G$-actions on $S^1$ which are minimal and strongly proximal.

\medskip\n
{\rm 2)} The set of measurable functions, up to equality almost everywhere,
\begin{equation*}
\omega : B^3 \r \IR
\end{equation*}
such that 
\begin{itemize}
\item[{\rm (1)}] $\omega$ is an alternating strict cocycle.
\item[{\rm (2)}] $\omega(g x, gy, gz) = \omega(x,y,z)$ for every $g \in G$ and almost every $(x,y,z) \in B^3$.
\item[{\rm (3)}] $\omega$ takes values in $\{\pm 1\}$ almost everywhere.
\end{itemize}
\end{theorem}

This bijection is implemented by the map which to a continuous minimal strongly proximal action $\rho: G \r \H(S^1)$ associates the cocycle
\begin{equation*}
\omega(x,y,z) = o\big(\varphi(x),\varphi(y),\varphi(z)\big)
\end{equation*}

\n
where $\varphi : B \r S^1$ is the map given by Corollary \ref{cor4.4}.

\medskip
The proof of the Theorem is divided in two steps. Fix $\omega : B^3 \r \IR$ satisfying the properties (1), (2), (3) above.

\bigskip\n
{\it \textus{Step 1:}} We show that there exists a measurable map $\varphi: B \r S^1$ satisfying the following properties:

\begin{itemize}
\item[(a)] $\varphi_*(\nu_B) \in \cM^1(S^1)$ has no atoms. 
\item[(b)] the essential image ${\rm Ess \,Im}\; \varphi$ of $\varphi$ equals $S^1$.
\item[(c)] $o\big(\varphi(x), \varphi(y), \varphi(z)\big) = \omega(x,y,z)$ for almost every $(x,y,z) \in B^3$.
\end{itemize}

\medskip\n
{\it \textus{Step 2:}} Given a measurable function $\varphi : B \r S^1$ satisfying (a), (b), (c) we construct a continuous homomorphism
\begin{equation*}
\pi_\varphi : G \r \H(S^1)
\end{equation*}

\n
such that $\varphi(gx) = \pi_\varphi(g) \big(\varphi (x)\big)$, for every $g \in G$ and almost every $x \in B$, and conclude that $\pi_\varphi$ is minimal, strongly proximal. Finally we show that given $\varphi,\psi$ satisfying (a), (b), (c) then $\pi_\varphi$ and $\pi_\psi$ are conjugate.

\medskip
Combining Corollary \ref{cor4.4}, Step 1 and Step 2 clearly completes the proof of Theorem \ref{theo4.5}.

\subsection{Step 1: the construction of the measurable map}

We fix once and for all a measurable map $\omega : B^3 \r \IR$ satisfying properties (1), (2), (3) in Theorem \ref{theo4.5}. For every $(x,y) \in B^2$ define
\begin{equation*}
I(x,y) = \{z \in B: \omega (x,z,y) = 1\}
\end{equation*}

\n
which is measurable. The next few lemmas are intended to show that the sets $I(x,y)$ behave like intervals on $S^1$. We will use the notation $E \equiv F$ to indicate that two sets $E,F$ differ by a set of measure zero; similarly $E \underset{\bullet}{\subset} F$ means that $E$ is contained \linebreak

\vspace{-3ex} \noindent in $F$ up to a set of measure zero.

\begin{lemma}\label{lem4.6} ~

\begin{itemize}
\item[{\rm (1)}] $I(x,y) \cup I(y,x) \equiv B$ for a.e. $(x,y) \in B^2$.
\item[{\rm (2)}] $I(x,y) \cap I(y,x) = \phi, \quad \forall (x,y) \in B^2$.
\item[{\rm (3)}] $\forall g \in G$ and a.e. $(x,y) \in B^2$, $g\big(I(x,y)\big) \equiv I(gx,gy)$.
\item[{\rm (4)}] For a.e. $(x,y) \in B^2$, $0 < \nu_B\big(I(x,y)\big) < 1$.
\end{itemize}
\end{lemma}

\begin{proof}
(1) follows from the fact that $\omega$ is alternating and $\omega(x,y,z) \in \{\pm 1\}$ for a.e. $(x,y,z) \in B^3$.

\bigskip\n
(2) is obvious.

\bigskip\n
(3) follows from the $G$-invariance of $\omega$ and the fact that the $G$-action on $B$ preserves the measure class of $\nu_B$.

\bigskip\n
(4) Consider

\vspace{-4ex}
\begin{align*}
A_1 & = \big\{(x,y) \in B^2: \;\nu_B \big(I(x,y)\big) = 1\big\}
\\[1ex]
A_0 & = \big\{(x,y) \in B^2: \;\nu_B \big(I(x,y)\big) = 0\} \,.
\end{align*}

\n
Denoting $\wedge$ the map $(x,y) \r (y,x)$, we have from (1) that $A_1 \equiv \widehat{A}_0$; in addition we have that $A_1 \cap A_0 = \phi$ and both sets are measurable and $G$-invariant by (3). Thus if $\nu^2_B(A_1) > 0$ then $\nu_B^2(A_0) = \nu^2_B(\widehat{A}_0) = \nu_B^2(A_1) > 0$ which contradicts the ergodicity of the $G$-action on $B^2$. Thus $A_1$ and $A_0$ are both of measure zero.
\end{proof}

\begin{lemma}\label{lem4.7}  {\rm 1)} For $(a,c) \in B^2$ and $b \in I(a,c)$ we have
\begin{align}
I(a,c)& \supset I(a,b) \cup I(b,c)\label{4.4}
\\[1ex]
I(a,b) & \cap I(b,c) = \phi \,.\label{4.5}
\end{align}

\n
{\rm 2)} For a.e. $(a,c) \in B^2$ and a.e. $b \in I(a,c)$
\begin{equation}\label{4.6}
I(a,c) \underset{\bullet}{\subset} I(a,b) \cup I(b,c) \,.
\end{equation}

\n
In particular, for a.e. $(a,c) \in B^2$ and a.e. $b \in I(a,c)$
\begin{equation}\label{4.7}
I(a,c)  \equiv  I(a,b) \cup I(b,c)\,.
\end{equation}
\end{lemma}

\begin{proof}
Applying the cocycle identity to $a,x,b,c \in B$ we have,
\begin{equation}\label{4.8}
\omega(x,b,c) - \omega(a,b,c) + \omega (a,x,c) - \omega (a,x,b) = 0\,.
\end{equation}

\n
If $b \in I(a,c)$ and $x \in I(a,b)$ we have $\omega(a,b,c) = 1$, $\omega(a,x,b) = 1$ which with (\ref{4.8}) implies $\o(x,b,c) + \o(a,x,c) = 2$ and hence $x \in I(a,c)$. If $x \in I(b,c)$ then (\ref{4.8}) gives $\o(a,x,c) - \o(a,x,b) = 2$ and hence $x \in I(a,c)$. This shows the inclusion (\ref{4.4}). Concerning (\ref{4.5}), we apply (\ref{4.8}) to $x \in I(a,b) \cap I(b,c)$ to get a contradiction.

\medskip
Concerning (\ref{4.6}), assume that $b \in I(a,c)$ and $x \in I(a,c)$ and apply (\ref{4.8}) to get
\begin{equation}\label{4.9}
\o(x,b,c) = \o(a,x,b)\,.
\end{equation}

\n
For a.e. $(a,b) \in B^2$ and a.e. $x \in B$ we have either $\o(a,x,b) = 1$ and hence $x \in I(a,b)$ or $\o(a,x,b) = -1$ which with (\ref{4.9}) gives $x \in I(b,c)$.
\end{proof}

Lemma \ref{lem4.6} (4) and Lemma \ref{lem4.7} imply then immediately

\begin{lemma}\label{lem4.8}
For a.e. $(a,x,y) \in B^3$ we have the following dichotomy:

\medskip\n
{\rm 1)} $x \in I(a,y)$, $I(a,y) \equiv I(a,x) \cup I(x,y)$ and $\nu_B\big(I(a,x)\big) < \nu_B\big((a,y)\big)$.

\medskip\n
{\rm 2)} $y \in I(a,x)$, $I(a,x) \equiv I(a,y) \cup I(y,x)$ and $\nu_B\big(I(a,y)\big) < \nu_B\big((a,x)\big)$.
\end{lemma}

Given $a \in B$, we define $f_a: B \r \IZ \backslash \IR$ by 
\begin{equation*}
f_a(x) : = \nu_B\big(I(a,x)\big) \; {\rm mod} \;\IZ \,.
\end{equation*}
Then we have:
\begin{lemma}\label{lem4.9}
For a.e. $a \in B$ we have:

\medskip\n
{\rm 1)} $E = \big\{(x,y) \in B^2: f_a(x) = f_a(y)\}$ is of measure zero.

\medskip\n
{\rm 2)} $(f_a)_* (\nu_B)$ has no atoms.
\end{lemma}

\begin{proof}
1) If $(x,y) \in E$ then either, $\nu_B\big(I(a,x)\big) = 0$ and $\nu_B\big(I(a,y)\big) = 1$, or $\nu_B\big(I(a,x)\big) = \nu_B\big((a,y)\big)$; in both cases these equalities hold only for a set of $(a,x,y)$'s of measure zero; in the first case this follows from Lemma \ref{lem4.6} (4) and in the second from Lemma \ref{lem4.8}.

\medskip\n
2)  follows from 1).
\end{proof}

\begin{lemma}\label{lem4.10}
For a.e. $a \in B$ and a.e. $(x,y,z) \in B^3$
\begin{equation*}
o\big(f_a(x), f_a(y), f_a(z)\big) = \omega(x,y,z)\,.
\end{equation*}
\end{lemma}

\begin{proof}
Observe that the left hand side gives a measurable alternating cocycle taking values in $\{-1,0,1\}$; it follows from Lemma \ref{lem4.9} that it takes values in $\{-1,1\}$ almost everywhere. It therefore suffices to show that if the left-hand side equals $1$ then so does the right-hand side.

\medskip
If $o\big(f_a(x), f_a(y), f_a(z)\big) = 1$ then up to cyclically permuting the variables we may assume, using the definition of $f_a$, that $\nu_B\big(I(a,x)\big) < \nu_B\big(I(a,y)\big) < \nu_B\big(I(a,z)\big)$. Together with Lemma \ref{lem4.8} this implies $x \in I(a,y)$ and $y \in I(a,z)$ except possibly for a set of $(a,x,y,z)$'s of measure zero. Using the cocycle identity for $\o$ applied to $(a,x,y,z)$ we get $\o(x,y,z) + \o(a,x,z) = 2$ and hence $\o(x,y,z) = 1$.
\end{proof}

The following lemma will be left to the reader:

\begin{lemma}\label{lem4.11}
Let $\xi \in \cM^1(S^1)$ be a probability measure without atoms and
\begin{align*}
h_\xi : &\; S^1 \r S^1
\\[1ex]
& \;x \longmapsto \xi\big([\dot{o},x)\big)\; {\rm mod} \;\IZ
\end{align*}
the associated quasiconjugacy. Then,

\begin{itemize}
\item[{\rm 1)}] $h_\xi$ is continuous.
\item[{\rm 2)}] The measure $(h_\xi)_* (\xi)$ has no atoms and its support equals $S^1$.
\item[{\rm 3)}] For $\xi^3$-almost every $(x,y,z) \in (S^1)^3$ we have $o\big(h_\xi(x),h_\xi(y), h_\xi(z)\big) = o (x,y,z)$.
\end{itemize}
\end{lemma}

Combining Lemma \ref{lem4.9}, \ref{lem4.10} and \ref{lem4.11} we obtain the following lemma which completes Step 1:

\begin{lemma}\label{lem4.12}
For a.e. $a \in B$, let $\xi = (f_a)_* (\nu_B)$ and define $\varphi_a : = h_\xi \circ f_a$. Then
\begin{itemize}
\item[{\rm (a)}] $(\varphi_a)_* (\nu_B)$ has no atoms.
\item[{\rm (b)}] ${\rm Ess\, Im} \;\varphi_a = S^1$.
\item[{\rm (c)}] $o\big(\varphi_a(x),\varphi_a(y),\varphi_a(z)\big) = \o(x,y,z)$ for a.e. $(x,y,z) \in B^3$.
\end{itemize}
\end{lemma}

\subsection{Step 2: the construction of the action}

The following result is the basis for our construction:

\begin{lemma}\label{lem4.13}
Let $\varphi,\psi: B \r S^1$ be measurable maps such that
\begin{itemize}
\item[{\rm 1)}] $\varphi_*(\nu_B)$ and $\psi_*(\nu_B)$ have no atoms.
\item[{\rm 2)}] ${\rm Ess\,Im} \;\varphi = {\rm Ess\, Im} \;\psi = S^1$.
\item[{\rm 3)}] $o\big(\varphi(x),\varphi(y),\varphi(z)\big) = o\big(\psi(x),\psi(y),\psi(z)\big)$
\end{itemize}
for a.e. $(x,y,z) \in B^3$.

\medskip
Then the essential image $F \subset S^1 \times S^1$ of the map
\begin{align*}
B& \r   S^1 \times S^1
\\[1ex]
 x & \longmapsto \big(\varphi(x), \psi(x)\big)
\end{align*}

\n
is the graph of an orientation preserving homeomorphism $h : S^1 \r S^1$ and, 
\begin{equation*}
\mbox{$h\big(\varphi(x)\big) = \psi(x)$ for a.e. $x \in B$}\,.
\end{equation*}
\end{lemma}

\begin{proof}
The proof is in two steps.

\medskip\n
{\it \textus{Claim 1:}} Given $(\xi_1,\eta_1), (\xi_2,\eta_2), (\xi_3,\eta_3)$ in $F$ such that $\xi_1,\xi_2,\xi_3$ are pairwise distinct and the same holds for $\eta_1,\eta_2,\eta_3$, we have
\begin{equation*}
o(\xi_1,\xi_2,\xi_3) = o(\eta_1,\eta_2,\eta_3) \,.
\end{equation*}

\n
Pick $V_i \ni \xi_i$ and $W_i \ni \eta_i$ open intervals such that $V_1,V_2,V_3$ are pairwise disjoint and the same holds for $W_1,W_2,W_3$. Then
\begin{equation*}
B_i = \big\{x \in B: \varphi(x) \in V_i \;\mbox{and} \;\psi(x) \in W_i\big\}
\end{equation*}
is of positive measure, $i = 1,2,3$.

\medskip
We have then for a.e. $(x,y,z) \in B_1 \times B_2 \times B_3$, 
\begin{align*}
o(\xi_1,\xi_2,\xi_3) & = o\big(\varphi(x_1), \varphi(x_2), \varphi(x_3)\big) = o \big(\psi(x_1), \psi(x_2), \psi(x_3)\big)
\\[1ex]
& = o(\eta_1,\eta_2,\eta_3)\,.
\end{align*}

\medskip\n
{\it \textus{Claim 2:}} Let $(\xi_1,\eta_1)$, $(\xi_2,\eta_2)$ be in $F$. Then $\xi_1 \not= \xi_2$ iff $\eta_1 \not= \eta_2$.

\medskip
Assume that $\xi_1 \not= \xi_2$. We will repeatedly use the fact that
\begin{equation}\label{4.10}
\big(\varphi(x), \psi(x)\big) \in F \;\;\mbox{for a.e. $x \in B$} \,.
\end{equation}

\n
Since ${\rm Ess \, Im} \;\varphi = S^1$ and $\xi_1 \not= \xi_2$ the set $\varphi^{-1}\big((\xi_1,\xi_2)\big)$ is of positive measure; then in virtue of (\ref{4.10}) and the hypothesis that $\psi_*(\nu_B)$ has no atoms, we can pick $t \in \varphi^{-1}\big((\xi_1,\xi_2)\big)$ with $\big(\varphi(t),\psi(t)\big) \in F$ and $\psi(t) \notin \{\eta_1,\eta_2\}$. Similarly we can find $s \in \varphi^{-1}\big((\xi_2,\xi_1)\big)$ such that $\big(\varphi(s),\psi(s)\big) \in F$ and $\psi(s) \notin \{\psi(t),\eta_1,\eta_2\}$. Then the cocycle identity applied to $(\eta_1,\psi(t), \eta_2,\psi(s)\}$ gives:
\begin{equation*}
o\big(\psi(t), \eta_2,\psi(s)\big) - o\big(\eta_1,\eta_2,\psi(s)\big) + o\big(\eta_1,\psi(t),\psi(s)\big) - o(\eta_1,\psi(t),\eta_2) = 0\,.
\end{equation*}

\n
Applying Claim 1 to the first and third terms we get:
\begin{equation*}
o\big(\varphi(t), \xi_2,\varphi(s)\big) - o\big(\eta_1,\eta_2,\psi(s)\big) + o\big(\xi_1,\varphi(t),\varphi(s)\big) - o(\eta_1,\psi(t),\eta_2) = 0
\end{equation*}
and by our choices we have
\begin{align*}
o\big(\varphi(t), \xi_2, \,\varphi(s)\big) & = 1
\\[1ex]
o\big(\xi_1,\varphi(t),\varphi(s)\big) & = 1
\end{align*}
which implies that:
\begin{equation*}
o\big(\eta_1,\eta_2,\psi(s)\big) +o(\eta_1,\psi(t),\eta_2) = 2
\end{equation*}
and hence $\eta_1 \not= \eta_2$.

\medskip
This shows Claim 2 modulo interchanging the roles of $\varphi$ and $\psi$.

\medskip
Since $pr_i(F) = S^1$ we get from Claim 2 that $F$ is the graph of a homeomorphism $h$ which by Claim 1 is orientation preserving. Finally (\ref{4.10}) says precisely that $h\big(\varphi(x)\big) = \psi(x)$ for a.e. $x \in B$ and this completes the proof of the lemma.
\end{proof}

Now we return to our cocycle $\o: B^3 \r \IR$ satisfying the hypothesis of Theorem \ref{theo4.5} and let $\varphi: B \r S^1$ be a measurable map as given by Lemma \ref{lem4.12}. Since $\o$ is $G$-invariant we can apply Lemma \ref{lem4.13}  to $\varphi$ and $\varphi g$, $g \in G$, to obtain an orientation preserving homeomorphism denoted $\pi_\varphi(g) \in \H(S^1)$ satisfying 
\begin{equation*}
\pi_\varphi(g)\big(\varphi(x)\big) = \varphi(gx) \;\; \mbox{for a.e. $x \in B$} \,.
\end{equation*}
This equivariance property implies that
\begin{equation*}
\pi_\varphi : G \r \H(S^1)
\end{equation*}
is a homomorphism.

\begin{lemma}\label{4.14}
{\rm 1)} $\pi_\varphi: G \r \H(S^1)$ is continuous. It is minimal, unbounded and strongly proximal.

\medskip\n
{\rm 2)} Let $\varphi,\psi$ be as in Lemma {\rm \ref{lem4.13}}, and $\pi_\varphi,\pi_\psi$ the corresponding homomorphisms. Let $h: S^1 \r S^1$ be the homeomorphism given by Lemma {\rm \ref{lem4.13}}. Then,
\begin{equation*}
h\,\pi_\varphi(g) = \pi_\psi(g)\,h \quad \forall g \in G\,.
\end{equation*} 
\end{lemma}

\begin{proof}
Assertion 2) is clear and follows from the various equivariance properties and thus we concentrate on 1).

\medskip
For a.e. $x \in B$, the map $g \longmapsto \pi_\varphi(g) \big(\varphi(x)\big) = \varphi(gx)$ is measurable and hence the homomorphism $\pi_\varphi: G \r \H(S^1)$ is measurable. Since $G$ is locally compact second countable and since $\H(S^1)$ is second countable we deduce that $\pi_\varphi$ is continuous. Using Proposition \ref{prop4.1} we see that $\pi^*_\varphi(e^b_\IR) \not= 0$ and hence $\pi_\varphi$ is non-elementary. Assume that there is an exceptional minimal set $ K \underset{+}{\subset} S^1$. Let $J$ be a connected component if $S^1 \backslash K$; since ${\rm Ess\, Im} \;\varphi = S^1$, $\varphi^{-1}(J)$ is of positive measure but not of full measure in $B$. Since $\forall g \in B$ we have either $\pi_\varphi(g)\, J = J$ or $\pi_\varphi (g) \,J \cap J = \phi$, we conclude that either $g \varphi^{-1}(J) \equiv \varphi^{-1}(J)$ or $g \varphi^{-1}(J) \cap \varphi^{-1}(J) = \phi$. Since the $G$-action on $B$ is ergodic we have $\bigcup_{g \in G} \,g \varphi^{-1}(J) = \varphi^{-1}(S^1 \backslash K) \equiv B$ and hence there is $g_0 \in B$ with $g_0^{-1} \,\varphi^{-1}(J) \cap \varphi^{-1}(J) = \phi$.

\medskip
Thus $\bigcup_{h \in G} h\big(\varphi^{-1}(J) \times \varphi^{-1}(J)\big)$ is of positive measure in $B^2$, $G$-invariant and does not meet $g_0 \,\varphi^{-1}(J) \times \varphi^{-1}(J)$, itself also of positive measure; this contradicts the ergodicity of the $G$-action on $B \times B$. Thus $\pi_\varphi$ is minimal unbounded. Finally it follows from Corollary \ref{cor4.4} that $\pi_\varphi$ is strongly proximal.
\end{proof}

\section{Proofs of Theorem \ref{theo1.2}, \ref{theo1.3}, Remark \ref{rem1.4} (2), Corollary \ref{cor1.5} and \ref{cor1.6}}
\setcounter{equation}{0}

In order to apply the results obtained so far we recall
\begin{theorem}\label{theo5.1} {\rm (\cite{Ka03})} 

\medskip
Let $G$ be a locally compact second countable group and $\mu \in \cM^1(G)$ a symmetric spread-out probability measure on $G$. Then the $G$-action on the associated Poisson boundary $(B,\nu_B)$ is amenable, doubly ergodic; the same properties hold for the action of a lattice $\Gamma < G$.
\end{theorem}

Using Theorem \ref{theo5.1} for $\Gamma = G$ and applying Corollary \ref{cor4.4}, Theorem \ref{theo4.5} and Proposition \ref{prop4.1} implies readily Theorem \ref{theo1.3}. Concerning Remark \ref{rem1.4} (2), if $\rho_1,\dots ,\rho_n$ are minimal strongly proximal actions and $\o_i : B^3 \r \{-1,1\}$ are the associated cocycle, then for any $n_i \in \IZ$, $\sum^n_{i=1} \, n_i \,\o_i$ takes integral values and hence by the isomorphism in (\ref{4.1}) and Proposition \ref{prop4.1} the norm of $\sum^n_{i=1} \,n_i \,\rho_i^*(e^b_\IR)$ is half-integral.

\medskip
Corollary \ref{cor1.6} (1) follows from Proposition \ref{prop3.2}, Proposition \ref{prop4.1} and Lemma \ref{lem4.3}.

\medskip
Corollary \ref{cor1.6} (2) follows from Theorem \ref{theo1.3} (2) which implies that $\rho_1,\rho_2$ have modulo conjugation the same strongly proximal quotient.

\medskip
We turn now to the proof of Theorem \ref{theo1.2}. First assume that for the given minimal unbounded action $\rho$, $\rho_{sp}$ extends continuously. Let $k = |\cZ_{\cH^+} (\rho(\Gamma))|$. Then we have that $(\rho_{sp})^* (e^b_\IR) = k \cdot \rho^*(e^b_\IR)$ (Lemma \ref{lem4.3}) and since $(\rho_{sp})^* \,(e^b_{\IR})$ is in the image of the restriction map, so is $\rho^*(e^b_\IR)$.

\medskip
Conversely, apply Theorem \ref{theo5.1} and let $\mu_\Gamma \in \cM^1 (\Gamma)$ be a probability measure of full support such that $\nu_B$ is $\mu_\Gamma$-stationary (see \cite{M91}). On the space $(B,\nu_B)$ the restriction map
\begin{equation*}
H^2_{b} (G, \IR) \r H^2_b(\Gamma, \IR),
\end{equation*}
is realized by the inclusion
\begin{equation}\label{5.1}
\cZ \,L_{\rm alt}^\infty(B^3,\IR)^G \hookrightarrow \cZ \,L^\infty_{\rm alt}(B^3,\IR)^\Gamma\,.
\end{equation}

\medskip\n
Now if $\rho^*(e^b_\IR)$ is in the image of the restriction map then so is $(\rho_{sp})^*(e^b_\IR)$.  Let $\o_{\rho_{sp}}: B^3 \r \IR$ be the cocycle associated to $\rho_{sp}$; using (\ref{5.1}) we get that $\o_{\rho_{sp}}$ is $G$-invariant, satisfies all hypothesis of Theorem \ref{theo4.5} and hence corresponds to a continuous minimal strongly proximal homomorphism $\pi: G \r \H(S^1)$. Since $\o_{\pi |_{\Gamma}}$ clearly coincides with $\o_{\rho_{sp}}$ we conclude from Theorem \ref{theo4.5} that $\pi |_{\Gamma}$ and $\rho_{sp}$ are conjugate.

\section{Locally compact groups acting on $\pmb{S^1}$}
\setcounter{equation}{0}

In this section we prove Theorem \ref{theo1.8} using some results from \cite{Gh01}, \cite{F01}. In the sequel, Rot denotes as usual the group of rotations of $S^1$ and, for every $k \ge 1$, PSL$(2,\IR)_k \subset \H(S^1)$ is the $k$-fold cyclic covering group of PSL$(2,\IR)$. We record the following

\begin{lemma}\label{lem6.1}
Let $H$ be connected, semi-simple, with finite center and no compact factors, and
\begin{equation}
\pi : H \r \H(S^1)
\end{equation}

\n
a continuous non-trivial homomorphism. Then there exists $k \in \IN$ such that up to conjugation, $\pi(H) = {\rm PSL}(2,\IR)_k$.
\end{lemma}

For the following lemma we recall that given a subgroup $L < \H(S^1)$, $\cN_{\cH^+}(L)$ and $\cZ_{\cH^+}(L)$ denote respectively its normalizer and centralizer in $\H(S^1)$.

\begin{lemma}\label{lem6.2} ~

\begin{itemize}
\item[{\rm 1)}] $\cN_{\cH^+} ({\rm Rot}) = \cZ_{\cH^+} ({\rm Rot}) = {\rm Rot}$.
\item[{\rm 2)}] $\cN_{\cH^+}({\rm PSL}(2,\IR)_k) = {\rm PSL}(2, \IR)_k$.
\item[{\rm 3)}] $\cZ_{\cH^+} ({\rm PSL}(2, \IR)_k) = \cZ ({\rm PSL}(2,\IR)_k)$, which is cyclic of order $k$.
\end{itemize}
\end{lemma}

\begin{proof}
1) Let $g \in \cN_{\cH^+}({\rm Rot})$; if $g \,r \,g^{-1} = r^{-1} \;\;\forall r \in {\rm Rot}$ then $g$ cannot preserve orientation; hence $g \in \cZ_{\cH^+} ({\rm Rot})$. Composing $g$ with a rotation we may assume $g(0) = 0$ which on applying $r \in {\rm Rot}$ yields $g\big(r(0)\big) = r(0)$ and thus $g = e$.

\bigskip\n
2) For $g \in \cN_{\cH^+}({\rm PSL}(2, \IR)_k)$ and $K_k < {\rm PSL}(2,\IR)_k$ maximal compact subgroup, $g^{-1} \,K_k \,g$ is maximal compact as well and hence there is $h \in {\rm PSL}(2,\IR)_k$ with $g\,K_k \,g^{-1} = h\,K_k \,h^{-1}$ that is $h^{-1} \,g \in \cN_{\cH^+}(K_k)$; since $K_k$ is conjugate to ${\rm Rot}$ we deduce from 1) that $h^{-1} g \in K_k$. 

\bigskip\n
3) follows from 2).
\end{proof}

Finally we record the following consequence of Hilbert's fith problem:

\begin{theorem}\label{theo6.3} {\rm (\cite{B-M02} Thm.~3.3.3)}

\medskip
Let $G$ be a locally compact group and $A(G) \vartriangleleft G$ its amenable radical. Let $G_a : = G/A(G)$. Then 

\medskip\n
{\rm 1)} $(G_a)^0$ is connected, semi-simple, with trivial center and no compact factors. 

\bigskip\n
{\rm 2)} The centralizer $L$ of $(G_a)^0$ in $G_a$ is totally disconnected; we have $L \cap (G_a)^0 = (e)$ and the product $L \cdot (G_a)^0$ is open of finite index in $G_a$.
\end{theorem}

\bigskip\n
{\it Proof of Theorem \ref{theo1.8}.} Let $\pi : G \r \H(S^1)$ be a continuous minimal action. We consider $\pi|_{A(G)}$, and apply our trichotomy:

\bigskip\n
{\it Case 1):} $\pi \big(A(G)\big)$ has a finite orbit; then the set of finite $\pi\big(A(G)\big)$- orbits with fixed cardinality is non-empty closed and $\pi(G)$-invariant, hence coincides with $S^1$. As a result, $\pi(A(G))$ is finite cyclic. 

\bigskip\n
{\it Case 2):} An exceptional minimal set for $\pi\big(A(G)\big)$ being unique would also be $\pi(G)$-invariant contradicting minimality of $\pi(G)$. 

\bigskip\n
{\it Case 3):} $\pi\big(A(G)\big)$ is minimal; since $A(G)$ is amenable $(\pi|_{A(G)})^* (e^b_\IR) = 0$ and hence, modulo conjugating $\pi$, we may assume that $\pi(A(G))$ is a dense subgroup of Rot. In particular, $\pi(G) \subset \cN_{\cH^+}({\rm Rot}) = {\rm Rot}$ by Lemma \ref{lem6.2}. This gives the first alternative in Theorem \ref{theo1.8}.

\medskip
Now we return to Case 1 and consider $G_\pi := G/(A(G)\cap {\rm Ker}\, \pi)$ which is a finite extension $q: G_\pi \r G_a$ of $G_a$. Observe that $\pi$ factors through $G_a$. Let $L_\pi : = q^{-1}(L)$; then $(G_\pi)^0$ is connected, semi-simple with finite center and no compact factors, $L_\pi$ is totally disconnected and $L_\pi \cdot (G_\pi)^0$ is open of finite index in $G_\pi$. We have then two cases:

\medskip\n
1) $\pi|_{(G_\pi)^0}$ is trivial. Since $\H(S^1)$ does not contain small subgroups, $\pi$ sends some compact open subgroup of $L_\pi$ to the identity which implies that ${\rm Ker}\, \pi$ is open.

\bigskip\n
2) $\pi |_{(G_\pi)^0}$ is non-trivial and hence by Lemma \ref{lem6.1}, $\pi\big((G_\pi)^0\big) = {\rm PSL}(2,\IR)_k$; since $\pi(G)$ normalizes $\pi\big((G_\pi)^0\big)$ we get $\pi(G) = {\rm PSL}(2,\IR)_k$ by Lemma \ref{lem6.2} (2). \hfill $\square$

\vfill
\noindent
Marc Burger\\
Forschungsinstitut f\"ur Mathematik\\
HG G 45.2\\
R\"amistrasse 101\\
CH-8092 Z\"urich
\\[1ex]
email: marc.burger@fim.math.ethz.ch
\end{document}